\documentclass[11pt]{article}
\usepackage{graphicx}
\usepackage{amssymb}
\usepackage{amsmath,amsthm}
\usepackage{verbatim}
\usepackage{amsfonts}
\newtheorem{thm}{Theorem}

\newtheorem{defn}[thm]{Definition}
\newtheorem{lem}[thm]{Lemma}

\newcommand{\E}{\mathcal{E}}

\newcommand{\G}{\mathcal{G}}

\newcommand{\Ta}{\mathcal{T}}

\newcommand{\bP}{\mathbb{P}}
\newcommand{\bE}{\mathbb{E}}

\newcommand{\bF}{\mathbb{F}}
\newcommand{\bN}{\mathbb{N}}

\title{Distinguishing Chromatic Number of Random Cayley graphs}
\author{Niranjan Balachandran\footnote{Supported by grant 12IRCCSG016, IRCC, IIT Bombay}, and Sajith Padinhatteeri \footnote{Supported by grant 09/087(0674)/2011-EMR-I, Council of Scientific \& Industrial Research, India},\\ Department of Mathematics\\
Indian Institute of Technology Bombay\\ Mumbai, India. }
\date{}
\begin{document}
\maketitle

\begin{abstract}
The \textit{Distinguishing Chromatic Number}  of a graph $G$, denoted $\chi_D(G)$, was first defined in \cite{collins} as the minimum number of colors needed to properly color $G$ such that no non-trivial automorphism $\phi$ of the graph $G$ fixes each color class of $G$. In this paper, we consider random Cayley graphs $\Gamma(A,S)$ defined over certain abelian groups $A$ and show that with probability at least $1-n^{-\Omega(\log n)}$ we have, $\chi_D(\Gamma)\le\chi(\Gamma) + 1$. 
\end{abstract}

\textbf{Keywords:}
Distinguishing Chromatic Number, Random Cayley graphs.\\

2010 AMS Classification Code: 05C15, 05C25, 05C80. 

\section{Introduction}

Let $G$ be a graph and let $Aut(G)$ denote its full automorphism group.  Albertson and Collins introduced the notion of the \textit{Distinguishing number of a graph} in \cite{AK}. 
\begin{defn} 
A labeling of vertices of a graph $G, h : V(G) \rightarrow  \{1, \hdots, r \}$
is said to be {\bf distinguishing} ( or $r$-distinguishing)  provided no nontrivial automorphism of the graph preserves all of the vertex labels. The distinguishing number of a graph $G,$ denoted by $D(G),$ is the minimum $r$ such that $G$ has an $r$-distinguishing labeling.\end{defn}

Collins and Trenk introduced the notion of the Distinguishing Chromatic Number in \cite{collins} as follows.
\begin{defn}
A labeling of vertices of a graph $G, h : V(G) \rightarrow  \{1, \hdots, r \}$ is said to be {\bf proper distinguishing} ( or proper $r$-distinguishing) 
provided the labeling is proper and distinguishing. The distinguishing chromatic number of a graph $G, \chi_{D}(G),$ is the minimum $r$ such that $G$ has a
proper $r$-distinguishing labeling. 
\end{defn}

In other words, the Distinguishing Chromatic Number of a graph $G$  is the least integer $r$ such that the vertex set can be partitioned into sets $V_1,V_2,\ldots, V_r$ such that each $V_i$ is independent in $G$, and for every $1\neq\pi\in Aut(G)$ there exists some color class $V_i$ such that $\pi(V_i)\neq V_i$. Since this notion is distinct from the notion of the chromatic number only when the graph admits non-trivial automorphisms, it is a matter of specific interest to determine the distinguishing chromatic number of graphs  with a large automorphism group.


 
 One class of graphs that decidedly admit non-trivial automorphism groups are Cayley graphs of groups. To recall the definition, let $A$ be a finite group with cardinality $n$ and let $S\subset A$ with $1 \notin S$ be an inverse closed subset of $A$. In other words, $S = S^{-1}$ where $S^{-1}:=\{g^{-1} : g \in S \}.$ The Cayley graph of $A$ with respect to $S,$ denoted by $\Gamma(A,S)$ is the following graph: $V(\Gamma(A,S)) = A$ and $E(\Gamma(A,S)) = \{(g, gh) : g \in A, h \in S \}.$   It is straightforward to see that the group $A$ acts regularly on $\Gamma(A,S)$.  If $A$ is abelian, the map $i(g) = -g$ is also an automorphism of $A$ which is distinct from any of the automorphisms induced by the member of $A$ unless $A\simeq\bF_2^r$ for $r\in \bN$. Hence, for $A$ abelian, it is easy to see that $A \rtimes \langle i\rangle\subset Aut(\Gamma(A,S)).$ In general, the full automorphism group of $\Gamma(A,S)$ can be larger, and determining it clearly depends on the  set $S$. 

In this paper we restrict our attention to random cayley graphs over classes of abelian groups, with the group operation expressed additively. The model for the random graphs on Cayley groups that we shall consider is described as follows. Let $A$ be a finite group with $|A|=n$, and let $0<p=p(n)<1$. Each element $g \in A$ of order $2$ is chosen with probability $p$ and for any other $x \in A,$ the pair $(x, -x)$ is chosen with probability $p$ and all these random choices are made independently to form the set $S$. The random Cayley graph is the graph $\Gamma_p:=\Gamma(A,S)$. 

Our main result in this paper states that $$\chi_D(\Gamma)\le\chi(\Gamma) + 1 \textit{\ with\  high\  probability\  (whp)}$$ over two classes of abelian groups which we shall describe below. Usually, the phrase `$E_n$ occurs with high probability' denotes that $\bP(E_n)\to 1$ as $n\to\infty$ for some relevant parameter $n$. In this paper, we shall also require rates of convergence, so our usage of the phrase shall mean that  $\bP(E_n) \ge 1-(n^{-c})$ for any constant $c$ for sufficiently large $n$. In our statements, $n=|A|$, the size of the underlying group. 

As mentioned earlier, if the group $A$ is abelian, then  $A \rtimes \langle i\rangle\subset Aut(\Gamma(A,S)).$ If equality holds in the above, then we say that $\Gamma(A,S)$ has automorphism group \textit{as small as possible}. A conjecture of  Babai and Godsil (see \cite{BG}) says that  if $A$ is an abelian group of order $n$, the proportion of inverse closed subsets $S$ for which for the corresponding Cayley graph $\Gamma(A,S)$  has automorphism group as small as possible is $1-o(1)$ as $n$ goes to infinity, and verified it when $n\equiv 3\pmod 4$. In a recent paper by Dobson, Spiga and Veret \cite{dobson}, the authors have proven this conjecture for all $n$. In other words, if an inverse closed set is picked uniformly at random then asymptotically almost surely, the corresponding random Cayley graph has automorphism group as small as possible.

In this paper we restrict our focus to the following families of random  Cayley graphs:
\begin{enumerate}
\item The random Cayley graph $\Gamma_p(A,S),$ with   $(|A|,6) = 1$. 
\item The random Cayley graph $\Gamma_p(A,S),$ where $A \cong \mathbb{Z}^r_2 \times N,$ and $N$ is an odd order group which is not cyclic. 
\end{enumerate}
We shall  call these as abelian groups of Type I and Type II respectively. Type I groups also appear in \cite{bgreen} where the chromatic number of a random Cayley graph $\Gamma_{1/2}(A,S)$ is determined asymptotically\footnote{In that paper, $x,y$ are adjacent in $\Gamma$ if and only if $x+y\in S$, and $S$ is picked uniformly at random from $A$.} As for Type II groups, the specific restrictions on $A$ may be relaxed, but it gets a bit messier to state the corresponding results, so we restrict our attention to these families of random Cayley graphs. 
\section{Preliminaries}
 Firstly, we show that the results of \cite{dobson} may be extended to the model of random graphs we are interested in, using very similar ideas, for a wider range of $p(n)$. We make no attempt to obtain the best possible constants that would make the following results work. We shall implicitly assume that $n$ is sufficiently large whenever the need arises.
 
  We write $f(n)\ll g(n)$ if $\displaystyle\lim_{n\to\infty}\frac{f(n)}{g(n)} = 0$. By $\log$ we shall mean $\log_2$ in the rest of this section. The number of elements of the group $A$ whose order is at most two, is denoted by $m$. 
\begin{lem}\label{lem00}
If $\frac{3}{2}\leq c_1,c_2\le n$ satisfy $c_1 c_2 \geq\frac{n}{24}$, and $p \in [\frac{25 (\log n)^2}{n}, 1 - \frac{ 25 (\log n)^2}{n}],$ then
$$n^{\log n} (p^{c_1} + (1-p)^{c_1})^{c_2}\leq n^{-\Omega(\log n)}.$$ 
\end{lem}
\begin{proof}
Consider $f(p) = (p^{c_1} + (1-p)^{c_1})^{c_2}$ on the interval $[0, 1].$ 
 It follows (by standard calculus) that $f(p)$ attains minimum at $p = \frac{1}{2}$.
 In particular,  for $p \in [\frac{25 (\log n)^2}{n}, 1 - \frac{ 25 (\log n)^2}{n}],$ $f(p)$ attains its maximum value at the endpoints. Therefore, it suffices to prove the statement when $ p = \frac{25 (\log n)^2}{n}$ since $f$ is symmetric about $1/2$.
 
Now,
$$ (p^{c_1} + (1-p)^{c_1})^{c_2} \leq e^{-pc_1c_2} \hspace{.1cm} e^{c_2 y^{c_1}}$$ where $y = \frac{p}{1-p}.$

For $p\in[\frac{25 (\log n)^2}{n},1/2]$, we first observe that for any $1<c_1,c_2\le n$, the expression $c_2y^{c_1}$ is bounded. Indeed, 
\begin{equation}\label{eq1}
c_2 y^{c_1} \leq c_2{\left(\frac{25(\log n)^2}{n - 25(\log n)^2}\right)}^{c_1}\le 1.
\end{equation}
The last inequality follows from the fact that $3/2\leq c_1,c_2\le n$. Therefore
$$ n^{\log n} (p^{c_1} + (1-p)^{c_1})^{c_2} \leq e^{(\log n)^2 - pc_1c_2}.$$

Since $c_1c_2\geq n/24$ the right hand side in the last inequality is at most  $\exp(-\frac{\log^2 n}{24})$. This completes the proof.
\end{proof} 

In what follows, unless otherwise mentioned, $p\in[\frac{25 (\log n)^2}{n},1-\frac{25 (\log n)^2}{n}]$. 
\begin{lem}\label{lem01}
Suppose $A$ is an abelian group which is not a $2$-group, and suppose $S$ is chosen randomly by picking each pair\footnote{If $x=-x$ then the pair is just the singleton $\{x\}$}  $(x,-x)$ independently
with probability $p$ where $\frac{25 (\log n)^2}{n} \leq p \leq \frac{1}{2}.$ Then,  
$$\bP(\textrm{There\  exist\ } 1 < H\leq K < A) \textrm{\ such\ that\ } S \setminus K  \textrm{is\  a\ union\  of\  }H-\textrm{cosets}) \leq O(\exp(-\log^2 n)) .$$
\end{lem}
\begin{proof}
Observe that since $H \subset K\subset A$,  $A\setminus K$ is also a union of $H-$ cosets, and let the set of these cosets be denoted $\mathcal{H}.$   Write $A':=A \setminus K$  and $S':=S \setminus K$. We shall denote the order of an element $a$ by $o(a)$ and $a+a$ is denoted $2a.$

Define \begin{eqnarray*} O_2 &:=& \{a \in A : o(a)\leq  2\},\\ J &:=& K \cap O_2,\\ I &:=& \{a \in A' :  2a \in H\},\\ I'&:=&A \setminus (K \cup I) ,\\ L &:=& \{a \in H : o(a) = 2\}.\end{eqnarray*}
Let $|H| =h,$ $|I| = i$,  $|K| = k,$ $|J| = j$  and $|L| = l.$ We have $|O_2| = m.$

The probability that $S'$ is a union of $H-$cosets is precisely 
$$\bP\left(\mathop{\bigcap}\limits_{g+H \in \mathcal{H}}\bigg\{ (g+H \subseteq S) \textrm{\ or\ } (g+H \cap S = \emptyset)\bigg\}\right).$$ 

 Let $g \in I'.$ If $h_1 \in H,$ and if possible $g + h_1\in K \cup I$ then, it follows that  $g + h_1 \in I$ so that $g \in I$ contradicting that $g\in I'$. Therefore, if $g \in I'$ then, $g+H \subseteq  I'.$ Moreover $-g\notin g + H $ which implies that $g+H\neq -g+H$. Also, observe that $I' \cap O_2 = \emptyset$.  Since each pair $(g, -g)$ is independently picked with probability $p$ into $S$ we have that $$g+H \subseteq S' \Longleftrightarrow -g+H \subseteq S'.$$ 
Since there are $\frac{n-k-i}{2h}$ pairs of cosets in $I'$ of the type $(g+H,-g+H)$, the probability that for every $g+H\in I'$ either $g+H\subset S$ or $g+H\cap S=\emptyset$ is exactly $(p^h + (1-p)^h)^{\frac{n-k-i}{2h}}.$

Suppose that $g \in I$. In this case note that $g+H = -g+H$. Suppose $o(g) =2.$ Then for $h \in H,$  we have $2(g+h) = 0$ if and only if $o(h) =2.$
In particular, the number of order $2$ elements in $g+H$ is precisely the number of order two elements in $H$. Since there are $l$ elements in $g+H$ of order two and $h-l$ elements of order greater than two, and since the number of $H$ cosets $g+H$ with $g\in I, o(g) = 2$ that contain order two elements is precisely $\frac{m-j}{l}$, the probability that every coset $g+H$ with $g\in O_2\cap I$ satisfies that $g+H\cap S=\emptyset$ or $g+H\subset S$ is precisely $(p^{\frac{h+l}{2}} + (1-p)^{\frac{h+l}{2}})^{\frac{m-j}{l}}.$

Finally, now suppose that $g \in I$ and $o(g) > 2$. In this case it follows that $g+H$ has no element of order two. There are exactly $i - \frac{m-j}{l}h$ elements $g\in I$ of this type and furthermore, the set of these elements must also necessarily be the union of $\frac{1}{h} (i - \frac{m-j}{l}h)$ $H-$cosets. If $g+H \subseteq S',$ one need to include the $\frac{h}{2}$ pairs $(x,-x)$ of the coset into $S$, so the probability that every $g+H$ with $o(g)>2$ is either disjoint with $S$ or is contained in $S$ is precisely $(p^{\frac{h}{2}} + (1-p)^{\frac{h}{2}})^{(\frac{i}{h} - \frac{m-j}{l})}.$

Again, as in the previous lemma, set $y:=\frac{p}{1-p}$. Then, from the above discussions, for a fixed
$H\subset K,$  we have,
\begin{align}
\bP(S' = \text{\ union\ of\ $H-$\ cosets}) &= (p^h +(1-p)^h)^{\frac{n-k-i}{2h}} (p^{\frac{h+l}{2}} +
(1-p)^{\frac{h+l}{2}})^{\frac{m-j}{l}} (p^{\frac{h}{2}} +(1-p)^{\frac{h}{2}})^{(\frac{i}{h} - \frac{m-j}{l})}\nonumber\\
&\leq  (1-p)^{\frac{n}{4}} \exp(\frac{n-k-i}{2h}y^h) \exp( \frac{m-j}{l}
y^\frac{h+l}{2}) \exp((\frac{i}{h} - \frac{m-j}{l}) y^\frac{h}{2})\nonumber
\end{align}
The last  inequality is obtained by using the facts that $k \leq \frac{n}{2}$ and $j \leq m.$  Furthermore, note that we may without loss of generality assume that $p \in [\frac{25 (\log n)^2}{n}, \frac{1}{2}].$  We shall now show that each of $\exp(\frac{n-k-i}{2h}y^h),
\exp( \frac{m-j}{l} y^\frac{h+l}{2}), \exp((\frac{i}{h} - \frac{m-j}{l}) y^\frac{h}{2})$ is bounded. 

If $h >2,$ then, using inequality (\ref{eq1}) of lemma \ref{lem00} and taking $c_1 = h, c_2 = \frac{n-k-i}{2h},$ it follows that $\exp(\frac{n-k-i}{2h}y^h)$ is bounded. 
Again using the same inequality, and taking $c _1 =\frac{h +l}{2} > 1$ and $c_2 =  \frac{m-j}{l} \leq n$ it follows that $\exp(\frac{m-j}{l} y^\frac{h+l}{2})$ is bounded.
As for $\exp((\frac{i}{h} - \frac{m-j}{l}) y^\frac{h}{2}),$ we set $c_1 = \frac{h}{2} >1$ and $c_2=(\frac{i}{h} - \frac{m-j}{l}) < n$. 
To pick a pair of non-trivial subgroups $H$ and $K,$ it suffices to only pick sets of generators for these groups which can be done in at most $(n^{\log n})^2=2^{2\log^2n}$ ways. Hence  
$$\bP(\text{There exist}\hspace{.1cm} 1 < H \leq K < A : |H|>2, S \setminus K =
\textrm{\ union\  of\  }H-\textrm{cosets}) \leq O\left(2^{2(\log n)^2}(1-p)^{\frac{n}{4}}\right).$$  By lemma \ref{lem00}, we have $2^{2(\log n)^2}(1-p)^{\frac{n}{4}} \leq \exp({-\frac{17}{4}(\log n)^2})$ for $p \in [\frac{25 (\log n)^2}{n}, \frac{1}{2}].$

If $h =2,$ then, firstly note that if $g$ satisfies $2g\in H$ then $o(g)|4$, so $g$ lies in the Sylow $2$-subgroup of $A$. Since $A$ is not a $2$-group by assumption, it follows that $i\le n/3$. Hence using that $j \leq m, k \leq \frac{n}{2}$ we have
  \begin{align}
\bP(S \setminus K \text{\ is a union of}\hspace{.05cm} H-\text{cosets}) &= (p^2 + (1-p)^2)^{\frac{n-k-i}{4}} (p^{\frac{2+l}{2}} +
(1-p)^{\frac{2+l}{2}})^{\frac{m-j}{l}} \nonumber \\
&\leq (1-p)^{\frac{n}{12}} \exp(\frac{n-k-i}{4} y^2 ) \exp(\frac{m-j}{l}y^{\frac{2+l}{2}} ) \label{eq11}
\end{align}
As before, the boundedness of  $\exp(\frac{n-k-i}{4} y^2 )$ follows by setting $c_1 = 2$ and $c_2 =\frac{n-k-i}{4} < n$ and the boundedness of $\exp(\frac{m-j}{l}y^{\frac{2+l}{2}} )$ follows by setting $c_1 = \frac{2+l}{2}> 1, c_2 = \frac{m-j}{l} < n$. Again, 
$$\bP(\text{There exist}\hspace{.1cm} 1 < H \leq K < A : |H|=2, S \setminus K =\textrm{\ union\  of\  }H-\textrm{cosets})\leq O\left(2^{2(\log n)^2}(1-p)^{\frac{n}{12}}\right)$$ and by lemma \ref{lem00}, this is at most $\exp({-\frac{1}{4}(\log n)^2})$. 
\end{proof}

The next lemma again is an extension of a result of \cite{dobson}. For $S\subset A$ and $\phi\in Aut(A)$, we say that \textit{$\phi$ normalizes S} if $\phi(S)=S$.
\begin{lem}\label{lem0}
Suppose $A$ is abelian, and let $S$ be a random inverse closed subset of $A$ with each pair $(x,-x)$ picked with probability $p$. Let  $i : A  \rightarrow A$ be  the automorphism of $A$ defined by $i: x \rightarrow -x.$ Then the probability that there exists $\phi \in Aut(A)\setminus \{1,i\}$ such that $S$ is normalized by $\phi$ is at most  $O(\exp(-\frac{21}{4} (\log n)^2).$ 
 \end{lem}
 \begin{proof}
  
  Fix $\phi \in Aut(A)$ and suppose that $\phi$ normalizes $S$. Since $| i | = 2,$ we have $m = |C_A(i)|$ where $C_A(i)$ is the centralizer of $i$ in A. Let $|C_A(\phi)| = c$ and  $| C_A(\langle i,\phi \rangle)| = k.$\\
 Suppose that $|\phi|$ is divisible by an odd prime $q.$\\
  In this case, without loss of generality we assume $|\phi|= q,$
otherwise we may replace $\phi$ with a suitable power. Observe that, if $ a \in S$ then $\{a, \phi(a), \hdots, \phi^{q-1}(a) \} \subseteq S.$ 
Therefore, 
$$\bP\big(\phi(S)\subset S \big) 
   = (p^q + (1 - p) ^ q )^{\frac{m - k}{q}}(p^q + (1 - p) ^ q)^{\frac{n- (c + m - k)}{2q}}\leq (p^q + (1 - p) ^ q )^{\frac{n}{4q}}.$$
  The last inequality follows by using $k \leq m, c \leq \frac{n}{2}$ and $(p^q + (1
- p) ^ q ) \leq 1.$ Since
$| Aut(A) | \leq n^{\log_2n},$ it follows that the probability that there exists $\phi \in Aut(A)\setminus\{1,i\}$ such that $\phi(S)=S$ is at most $n^{\log_2n} (p^q + (1 - p)
^ q )^{\frac{n}{4q}}.$ We use lemma \ref{lem00}, by setting $c_1 = q$ and $c_2 = \frac{n}{4q}$  to see that this probability is $O(\exp(-\frac{21}{4} (\log n)^2).$

Now suppose $|\phi|$ is a power of two. Two cases arise:\\
\textbf{\textit{Case 1:}} $ i \in \langle\phi \rangle$ \\
  By replacing $\phi$ by a suitable power, we may assume that $\phi^2 = i.$ Then, similar to the case $1,$
$$ \bP(\phi(S)\subset S) = (p^2 + (1-p)^2)^{\frac{m-c}{2}}(p^2 +
(1-p)^2)^{\frac{n-m}{4}}\leq (p^2 + (1-p)^2)^{\frac{n}{8}}.$$
The last inequality is obtained by using $m \leq \frac{n}{2}$ and $ c \leq m.$
Again, we use lemma \ref{lem00} with 
$c_1 = 2$ and $c_2 = \frac{n}{8}$ to see that the above probability is at most
 $2^{(\log n)^2}(p^2 + (1-p)^2)^{\frac{n}{8}}\leq O(\exp(-\frac{21}{4} (\log n)^2).$
  
\textbf{Case 2:} $i \notin \langle\phi \rangle$ \\ 
In this case  
$$   \bP(\phi(S)\subset S) = (p^2 + (1-p)^2)^{\frac{m}{2}}(p^2 +
(1-p)^2)^{\frac{n-m}{4}} = (p^2 + (1-p)^2)^{\frac{m+n}{4}}.$$
Again, setting $c_1 = 2, c_2 = \frac{m+n}{4}$ and applying lemma \ref{lem00} we see that the above probability is at most $ O(\exp(-\frac{23}{2} (\log n)^2)$. \end{proof}
 
 The following lemma is proved for the abelian groups of Types I and II.
\begin{lem}\label{lem03} Let $A$ be an abelian group of  Type I or Type II. Let $C$ be a cyclic group, and $Z$ an elementary abelian $2$ group. For a subset $S\subset A$, we call a pair of subgroups $(C, Z)$ of $A$,  good for $S$, if 
\begin{enumerate}
\item $A = C \times Z.$ 
\item $|C|=t\geq 4.$  
\item There exist $S' \in \{C, \emptyset, \{0\}, C \setminus \{0\}\},$ and $S'' \subset Z$ such that $S = S' \times S''$. \end{enumerate}
For a random inverse-closed subset $S\subset A$,  the probability that there exists a pair $(C,Z)$ good for $S$ is at most $O\left(\exp(-\frac{25 (\log n)^2(n-1)}{2n})\right).$
\end{lem}
\begin{proof}
The lemma is trivial in the case $A \cong \mathbb{Z}^r_2 \times N,$ where $N$ is an odd order group which is not cyclic. Let $A$ be  abelian with $(|A|,6)=1$. For a fixed $S\subset A$ which is inverse-closed, if $(C,Z)$ is good for $S$, then $A \cong C$, $Z$ is trivial, and furthermore, $S' \in \{\emptyset, A, (0), A \setminus \{0\}\}, S'' \in \{ \emptyset, (0) \}.$  Since $S$ is inverse-closed and $0 \notin S,$ there are two possibilities: $S=\emptyset$ or $S=A\setminus\{0\}$. In either case, it is easy to check that   
 the probability that there exist $(C, Z, S', S''),$ satisfying the hypotheses is at most $\exp(-\frac{25 (\log n)^2(n-1)}{2n}).$
\end{proof}

 For the abelian groups mentioned in the beginning of this section, we state the extended version of Theorem 1.5 from \cite{dobson}. The proof is along the same lines as the proof that appears in \cite{dobson}, so we skip the details.

\begin{thm}\label{smallaut}
 Let $\Gamma_p:=\Gamma_p(A,S)$ be the random Cayley graph with $\frac{25(\log n)^2}{n} \leq p \leq 1- \frac{25(\log n)^2}{n}.$ Then, $$\bP(Aut(\Gamma_p) \not\cong A \rtimes \langle i \rangle ) \leq O(\exp(-\log^2 n),$$ where $i:A\to A$ is the automorphism $i(x) = -x$. \end{thm}
\section{Random Cayley graphs on Type I groups}

We first consider abelian groups $A$ with $(|A|,6)=1$. Set $|A|=n$. We adopt the convention that an event $\E$ occurs in the random Cayley graph $\Gamma_p(A,S)$ \textit{with high probability (whp for short)} if $\bP(\E)\geq 1-n^{-\Omega(\log n)}$. 
\begin{thm}\label{jan} Let $\Gamma_p:=\Gamma_p(A,S)$ be the random Cayley graph with $$\frac{25(\log n)^2}{n} \leq p \leq 1- \left(\frac{10\log n}{n}\right)^{2/3},$$ where $A$ is an abelian group with order co-prime to six. Then, $\chi_D(\Gamma)\leq\chi(\Gamma)+1$ with high probability.
\end{thm}
\begin{proof}  Our main probabilistic tool here is Janson's inequality. To set the notation up, we first give the setup and state Janson's inequality.

Let $R\subset\Omega$ be a random subset where each $r \in \Omega$ is chosen into $R$ independently with probability $p_r$. Let $X_i\subset\Omega$ for $i=1,2\ldots, t$, and let $B_i$ denote the event: $X_i\subset R$.  Let $$N=\#\{i : X_i\subset R\}, \hspace{0.3cm}\mu:=\bE(N),\hspace{0.3cm} \Delta:=\sum_{i \sim j} \bP(B_i \wedge B_j),$$  where $i\sim j$ if $X_i\cap X_j\neq\emptyset$. Then, $$\bP(N =0) \leq \exp\left(-\frac{\mu^2}{2\Delta}\right)$$ if $\mu \leq\Delta.$

The random process of picking $S$ is equivalent to rejecting each pair $(x,-x)$ in $A$ (for $x\neq 0$) independently with probability $q=1-p$. 

 Let $\Ta := \{\{x, y, z\}\subset A: x + y + z = 0, x \neq 0, y \neq 0, z\neq 0\}$ and for each $T \in \Ta,$ let $$D(T) := \{\pm(x-y), \pm(y-z), \pm(x-z)\}.$$ 
 First, observe that $|\Ta|=\frac{(n-5)(n-1)}{6}$. Indeed, there are $n-1$ choices for $x$ with $x\neq 0$, and since $y \notin \{0, x, -x, 2x \}, 2y \neq -x,$ there are $n-5$ choices for $y$ and $z$ is consequently determined uniquely, so that gives $(n-1)(n-5)$ ordered triples $(x,y,z)$ satisfying the conditions of the sets in $\Ta$.
 
 Consider the events $B_T$: $D(T) \subset \overline{S} ,$ and let $N = \# \{T\in\Ta: D(T) \subset \overline{S}\}$. Then 
 $$ \bE(N) =  | \Ta | q^3 = \frac{(n-5)(n-1)}{6}q^3.$$
Observe that, $T \sim U$ if and only if $| D_T \cap D_U | \neq 0$ since otherwise the choices for the sets $T,U\in\Ta$ are decided over disjoint sets of inverse-closed pairs. 
Set
\begin{align}
 \Delta &= \sum_{| D(T) \cap D(U) | \neq 0} \bP(B_T \wedge B_U)  \label{delta}
\end{align}
We shall find a suitable upper bound for $\Delta$ and in order to do that, we shall count the number of $U \in \Ta$ with $U\sim T$ for a fixed  $ T \in \Ta$.


Suppose that $|D(T)\cap D(U)|=2$.  Let $T = \{x,y,z\}$ and $U = \{u, v, w\}.$  If one of $(x-y), -(x-y) \in U$, say $x-y = u-v,$ then it follows that $\{u, v, w \} = \{u, u - (x-y), -2u + (x-y)\}$ for some $0\neq u \in A.$  In particular, for a given $T \in \Ta$ there are $3(n-1)$ choices for $U$ such that $|D(T)\cap D(U)|=2.$ One can check (by a straightforward calculation; we skip the details) that there is at most one set $U$ with  $| D(T) \cap D(U) | = 4$, and that for any $T\in\Ta$, $-T\neq T$, and $U=-T$ is the unique member of $\Ta$ satisfying $|D(T)\cap D(U)|=6$.

Therefore, we have
$$ \Delta < 3n| \Ta | q^5 + | \Ta | q^4 + | \Ta | q^3,$$ so
by Janson's inequality, it follows that 
$$ \bP(N = 0)  <  \exp\left({\frac{-| \Ta | q^3}{2(3n q^2 + q + 1)}}\right)=e^{-\Omega(\log^2 n)}$$
for $q \ge \left(\frac{17\log n}{n}\right)^{2/3}.$ 

  Suppose $\sigma \in A \rtimes \langle i \rangle$ is non-trivial and $\sigma(T) = T$ for some $T \in \Ta.$ If $\sigma = (g,1)$ for some $g \in A$, and if $\sigma(x) = y, \sigma(y) = z, \sigma(z) = x,$ say, then by the action of $(g,1)$ on $A$, it follows that $3g = 0$ contradicting that $\sigma$ is non-trivial. If  $\sigma(x) = y, \sigma(y) = x$ and $\sigma(z) = z$, say,   Then, it similarly follows that $2g = 0$, contradicting that $\sigma$ is non-trivial.  If $\sigma = (g,i)$ for some $g \in A$, and if $\sigma(x) = y, \sigma(y) = z$ and $\sigma(z) = x,$ then
since $(g,i)(x) = g - x$, it follows that  $x = y = z$ contradicting that $\{x,y,z\}\in\Ta.$ Again, if $\sigma(x) = y, \sigma(y) = x$ and $\sigma(z) = z.$ Then, it follows that  $2z - x = y$ and since $x + y + z =0,$ we have $2z = 0$, again, a contradiction to the assumption that $\{x,y,z\}\in\Ta$. The upshot is that no non-trivial $\sigma\in A \rtimes \langle i \rangle$ fixes any $T\in\Ta$.

By theorem \ref{smallaut}, the full automorphism group of this random Cayley graph is isomorphic to $A \rtimes \langle i \rangle$ \textit{whp}. From the preceding discussions, it follows that the random Cayley graph $\Gamma_p(A,S),$ contains a $3$-element independent set $\{x,y,z\}$ which is not fixed by any non-trivial automorphism $\sigma \in Aut(\Gamma)$ \textit{whp}. Color this set with a new color and the rest of the graph using as few colors as possible. This coloring is both proper and distinguishing.   
\end{proof}

\section{Random Cayley graphs on Type II groups}
The next theorem deals with the other case of abelian groups as indicated in the beginning of this section. Firstly we shall need a general lemma. To prove the lemma we use a variant of the motion lemma \cite{ns1}. As a completion we state a variant of the motion lemma.

\begin{lem}[A variant of the motion lemma]\hfill \break\label{prarg}
Let $C$ be a proper coloring of the graph $G$ with $\chi(G)$ colors and let $C_1$ be a color class in $C$. Let $\G$ be the subgroup of $Aut(G)$ consisting of all automorphisms that fix the color class $C_1$. For each $A\in\G$, let $\theta_A$ denote the total number of distinct orbits induced by the automorphism $A$ in the color class $C_1$. If for some integer $t\geq 2$, $$f(\G) = \sum_{A \in \G } t^{\theta_A-|C_1| }< r$$ where $r$ is the least prime dividing $|\G|$,
 then $\chi_D(G) \leq \chi(G) + t-1$. In particular, if $F(C_1) < |C_1| - 2 \log_t|\G|$ then this conclusion holds, where $F(C_1)$ is the maximum number of vertices a nontrivial automorphism can fix in $C_1$.
\end{lem} 

\begin{lem}\label{z2n}
Let $A\simeq\mathbb{Z}^r_2 \times N,$ where $N$ is a non-cyclic group of  odd order and let $\Gamma=\Gamma(A,S)$ be a Cayley graph on $A$. Suppose that $Aut(\Gamma) \cong  A \rtimes \langle i \rangle.$  If  $m$ is the number of elements in $A$ of order at most $2,$ and $\chi(\Gamma) < \frac{n}{m+2 \log (2n)},$ then $\chi_D(\Gamma)\leq \chi(\Gamma) + 1.$
\end{lem}

\begin{proof} Let us denote $\chi(\Gamma)=\chi$ and let $C_1$ be a maximum sized color class in a proper coloring of $\Gamma$ using $\chi$ colors, so that $| C_1| \geq n/\chi.$  


Observe that a non-trivial automorphism which fixes any vertex of $\Gamma $ is necessarily of the form $(g, i)$ for some $g \in A.$ Moreover, $(g,i)$ fixes a vertex $h \in \Gamma$ if
and only if $g =2h$ in $A.$ It follows that any non-trivial automorphism $\sigma$ fixes at most $m$ vertices in $\Gamma.$ Therefore in Lemma \ref{prarg}, $\theta_{\sigma} \le m + (|C_1| - m )/2$. Therefore $f(\G) \le 2nt^{-\alpha}$ where $\alpha:=\frac{n/\chi - m}{2}.$ Now observe that $$t:=\lceil(2n)^{\frac{2\chi}{n-m\chi}}\rceil \Longrightarrow 2n<t^{\alpha}.$$
Hence there exists a proper $\chi + t-1$ coloring of $\Gamma$ that is also distinguishing. In particular,  if $\chi < \frac{n}{m+2\log (2n)}$ we may take $t=2$, and this proves the lemma. \end{proof}

Finally we have the corresponding theorem for random Cayley graph $\Gamma_p(A,S)$ for $A\simeq\mathbb{Z}^r_2 \times N$ with $N$ being a non-cyclic group of odd order.
\begin{thm}\label{finthm} Suppose $A$ is a Type II abelian group of order $n$ and suppose  that  $m \ll \frac{n}{\log^2 n}$. Let $\Gamma_p:=\Gamma_p(A,S)$ be the random Cayley graph, with $\frac{25(\log n)^2}{n} \leq p \leq \frac{7}{13(m+2\log 2n)}$ .  Then \textit{whp} $$\chi_D(\Gamma_p)\le \chi(\Gamma_p) + 1.$$ 
\end{thm}

\begin{proof}
Let $$X' := \sum\limits_{\substack{x:2x=0 \\ x \neq 0}} \mathbf{1}_{x\in S} \hspace{1cm}  X'':= \sum\limits_{\substack{(x, -x )\\{x \neq
-x}}} \mathbf{1}_{x,-x\in S}$$ so $|S|=X'+2X''$. Then $X', X''$ are  binomial random variables with parameters $(m-1, p)$ and $(\frac{n-m}{2},p)$ respectively.  Then
$$
 \bE(|S|) =(n-1)p < np.
$$
By the concentration of binomial random variables (see theorem 2.1 in \cite{janson}) we have 
\begin{align}
 \bP(|S| \geq \bE(|S|) + 3t)  &\leq \bP(X' \geq \bE(X') + t) + \bP(X'' \geq \bE(X'') + t) \nonumber \\
 &\leq\exp \Big(-\frac{t^2}{2((m-1)p + \frac{t}{3})}\Big) + \exp\Big(-\frac{t^2}{2(\frac{n-m}{2}p + \frac{t}{3})}\Big)\label{chernoff3}\end{align}

Set $t = \frac{2n}{13(m+2 \log2n)}.$ Since $m\ll \frac{n}{\log^2 n}$ it follows that for $$\frac{25\log^2 n}{n}\leq p <\frac{7}{13(m+2\log 2n)} <1-\frac{25\log^2 n}{n}$$ the right hand side of (\ref{chernoff3}) is at most $e^{-\Omega(\log^2 n)}$, so that \textit{whp} $|S|\le \frac{13np}{7}<\frac{n}{m+2 \log (2n)}$. Hence by theorem \ref{smallaut} and lemma \ref{z2n}, and the fact that $\chi(G)\le \Delta(G)+1$ for any graph $G$, it follows that $\chi_D(\Gamma_p)\le\chi(\Gamma_p)+1$ \textit{whp}. \end{proof}

\section{Concluding Remarks}
\begin{enumerate}
\item As emphasized in the introduction, all our results regarding random Cayley graphs hold with probability $1-n^{-\Omega(\log n)}$. However, if we wish to only prove that  results asymptotically almost surely, i.e., with probability $1-o(1)$, then improvements on some of the results is not difficult. For instance, Alon proved in \cite{Al} that if we pick $k\le n/2$ subsets uniformly at random and then complete them to inverse-closed sets, then \textit{a.a.s} $\chi(\Gamma(A,S))\le O\left(\frac{k}{\log k}\right)$. So for $A\simeq\mathbb{Z}^r_2 \times N$ with $N$ a non-cyclic group of odd order with $n^{3/4}\log n\ll m\ll  \frac{n}{\log n}$, one can prove by minor modifications, that \textit{a.a.s} $\chi_D(\Gamma_p)\le\chi(\Gamma_p)+1$ if $\frac{c\log^2 n}{n}\le p\le \frac{C\log n}{m+2\log 2n}$ for suitable constants $c,C$. We skip the details.
\item It is possible to extend some of the methods in the study of $\chi_D(\Gamma_p(A,S))$  to other abelian groups as well.  For non abelian groups $A$, it is a yet-unsettled conjecture of  Babai, Godsil, Imrich, and Lov\'asz (see \cite{BG} for details and a proof of the conjecture for nilpotent non-abelian groups), that for any group which is not generalized dihedral, almost surely $Aut(\Gamma_{1/2}(A,S))\simeq A$ as $|A|\to\infty$. Thus, for all such graphs it is clear that $\chi_D(G)\le \chi(G)+1$ since one can pick an arbitrary non-identity vertex and color it using a distinct color, and color the rest of the graph using at most $\chi(G)$ colors. Since $A$ acts regularly, it follows that this coloring is distinguishing as well.  We in fact believe that something stronger is true, viz., that for almost all Cayley graphs, $\chi_D(G)=\chi(G)$. At the moment, we are only able to show the same in certain non-abelian $q$-groups, for $q$ a large enough prime. Indeed, by the result of \footnote{It requires a very small tweak but the proof runs through without any major changes} \cite{BG}, for $p=1/2$, a random Cayley graph $\Gamma=\Gamma_{1/2}(A,S)$ almost surely has full automorphism group isomorphic to $A$, when $A$ is a  nilpotent non-abelian group. Furthermore, from a result\footnote{Again, the proof in \cite {Al} can be followed as it is in our random Cayley graph model to get the same result.} of \cite{Al}, we have $\chi(\Gamma)=\Omega(\frac{n}{\log^2 n})$. Suppose $|A|=q^r$ for a fixed $r$, and $q$ a  sufficiently large prime. If $\phi=\phi_g$ for $g\in A$ is an automorphism that fixes every color class of this coloring, then note that each color class has at least $q$ elements, so that $\chi(\Gamma)\le q^{r-1}$. But this contradicts the result of \cite{Al} since $q\gg \Omega_{r}(\log^2 q)$. The same arguments work over a slightly larger range for $p=\Omega(1)$ along the same lines as discussed above.
\end{enumerate}

 \end{document}